\documentclass[12pt,oneside]{amsart}

\usepackage{geometry}
\usepackage[english]{babel}
\usepackage{amscd,amssymb,amsthm,amsmath,hyperref,tikz-cd,stmaryrd,mathrsfs}
\usepackage[matrix,arrow,curve]{xy}
\usetikzlibrary{babel}
\usepackage{cleveref}

\DeclareFontFamily{U}{mathx}{\hyphenchar\font45}
\DeclareFontShape{U}{mathx}{m}{n}{
      <5> <6> <7> <8> <9> <10>
      <10.95> <12> <14.4> <17.28> <20.74> <24.88>
      mathx10
      }{}
\DeclareSymbolFont{mathx}{U}{mathx}{m}{n}
\DeclareFontSubstitution{U}{mathx}{m}{n}
\DeclareMathAccent{\widebar}{0}{mathx}{"73}

\tikzset{%
  symbol/.style={
    draw=none,
    every to/.append style={
      edge node={node [sloped, allow upside down, auto=false]{$#1$}}
    },
  },
}

\newtheorem{theorem}[subsection]{Theorem}
\newtheorem{lemma}[subsection]{Lemma}
\newtheorem{corollary}[subsection]{Corollary}
\newtheorem{proposition}[subsection]{Proposition}

\theoremstyle{definition}
\newtheorem{definition}[subsection]{Definition}
\newtheorem{remark}[subsection]{Remark}

\newtheorem{notation}[subsection]{Notation}

\newcommand{\CC}{\mathbb{C}}

\newcommand{\PP}{\mathbb{P}}

\newcommand{\ZZ}{\mathbb{Z}}
\newcommand{\GG}{\mathbb{G}}

\newcommand{\AAA}{\mathbb{A}}
\newcommand{\OOO}{\mathcal{O}}
\newcommand{\NNN}{\mathcal{N}}

\newcommand{\Pic}{\operatorname{Pic}}
\newcommand{\Sing}{\operatorname{Sing}}

\newcommand{\Bl}{\operatorname{Bl}}
\newcommand{\G}{\operatorname{G}}
\newcommand{\Aut}{\operatorname{Aut}}

\newcommand{\GL}{\operatorname{GL}}

\title{Cylinders in Fano threefolds of genus 9 and 10}
\author{Nikita Virin}
\address{Steklov Mathematical Institute of Russian Academy of Sciences}
\email{virinnikita@gmail.com}
\subjclass{Primary 14J45; Secondary 14N20, 14R20.}

\begin{document}

\begin{abstract}
   We prove that any Fano threefold of genus $9$ and $10$ contains a cylinder, i.e. an open subset isomorphic to the product of a quasiprojective variety and the affine line. Moreover, we show that any Fano threefold of genus $10$ has a point such that the Hilbert scheme of lines through the point has length three.
\end{abstract}

\keywords{
    Fano variety, cylinder.
}

\thanks{This work was performed at the Steklov International Mathematical Center and supported by the Ministry of Science and Higher Education of the Russian Federation (agreement no. 075-15-2025-303).}

\maketitle

\section{Introduction}

We assume that all varieties are defined over the complex number field $\CC$. We recall that a \emph{cylinder} is an algebraic variety $U$ that is isomorphic to $Z\times\AAA^1,$ where $Z$ is a quasiprojective variety. A variety $X$ is called \emph{cylindrical} if it contains a Zariski open subvariety $U$ that  is a cylinder. We refer to \cite{che2} for a survey about cylinders in Fano varieties.

Interest in cylindrical varieties originates from affine geometry. Let $X$ be a projective variety and let $X\hookrightarrow \PP^N$ be its embedding into a projective space. In \cite{kis3} the following question was raised:

When does an affine cone $\operatorname{AffCone}(X)\subset\AAA^{N+1}$ over $X$ admit an effective action of the additive group $\GG_a$? 

Note that if the affine cone $\operatorname{AffCone}(X)$ admits such an action and $\dim(X)\geqslant 1$, then the automorphism group $\Aut(\operatorname{AffCone}(X))$ is infinite-dimensional (see \cite{arz}).

One necessary and sufficient condition is the following.

\begin{theorem}[{\cite[Corollary 0.4]{kis}}]
    \label{ac}
    Suppose that $X\subset\PP^N$ is a smooth projective variety with $\Pic(X)=\ZZ$. Then $\operatorname{AffCone}(X)$ admits an effective $\GG_a$-action if and only if $X$ is cylindrical.
\end{theorem}

    There is an analogous necessary and sufficient condition for varieties with an arbitrary Picard group (see \cite[Corollary 3.2]{kis3}).

Any cylindrical variety is uniruled, hence its Kodaira dimension is negative (see \cite[Corollary IV.1.11]{kol}). Therefore, any smooth cylindrical variety with $\operatorname{rk}\Pic(X)=1$ is a Fano variety.

It is known that cylindrical Fano threefolds are rational (see \cite{kis}). But there are no known examples of non-cylindrical smooth rational Fano varieties of Picard rank one.

Let us recall that the \emph{genus} of a Fano threefold with $\Pic(X)= \ZZ[-K_X]$ is the number $$g:=1-\frac{K_X^3}{2}.$$

It is known that this number is an integer and takes values in the interval $2\leqslant g\leqslant 12$, and $g\neq 11$ (see e.g. \cite[Theorem 4.3.3]{isk2}).

Assume that $X$ is a smooth Fano threefold of genus $g$ with $\Pic(X)= \ZZ[-K_X]$. 
Let us recall known facts about the rationality and the cylindricity of such varieties:
\begin{enumerate}
    \item 
If $g\in\{2, 3, 4, 5, 6, 8\}$, then a general variety $X$ is not rational, in particular, is not cylindrical. More precisely:
\begin{itemize}
    \item If $g\in\{2, 3, 5, 8\}$, then any variety $X$ is not rational (see \cite{bea}, \cite{cle}, \cite{isk6}, \cite{isk4}, \cite{isk5}).
    \item If $g\in\{4, 6\}$, then a general variety $X$ is not rational (see \cite{bea}, \cite{isk5}, \cite{tyu}).
\end{itemize}
\item If $g\in\{7, 9, 10, 12\}$, then any variety $X$ is rational (see \cite{isk2}).
\item
If $g=12$, then any variety $X$ is cylindrical (see \cite[Proposition 5.2]{kis2}).
\end{enumerate}

There remains a question on the cylindricity of Fano threefolds of genus $7$, $9$, and $10$ with $\Pic(X)= \ZZ[-K_X]$. It has been known that in the class of Fano threefolds of genus $9$ and $10$ with $\Pic(X)= \ZZ[-K_X]$ there are cylindrical varieties. More precisely, in \cite{kis} the following theorem is proved.

\begin{theorem}[{\cite[Theorem 0.1]{kis}}]
\label{st}
Suppose that $X$ is a smooth Fano threefold of genus $g=9$ or $g=10$ with $\Pic(X)=\ZZ[-K_X]$. If the Hilbert scheme of lines $\mathrm{F}_1(X)$ is singular, then $X$ is cylindrical. All such Fano varieties form a codimension one family in the corresponding moduli space.
\end{theorem}
In the present paper, we prove the cylindricity of an arbitrary Fano threefold of genus $9$ or $10$ with $\Pic(X)= \ZZ[-K_X]$. Namely, we prove the following theorem:

\begin{theorem}
\label{mt}
Let $X$ be a smooth Fano threefold of genus $9$ or $10$ with $\Pic(X)=\ZZ[-K_X]$. Then $X$ is cylindrical.
\end{theorem}

If $X$ is a smooth Fano threefold of genus $9$ or $10$ with $\Pic(X)=\ZZ[-K_X]$, then $\Aut(X)$ is finite (see \cite{pr1}), in particular, $X$ does not admit effective $\GG_a$-actions, and therefore $\operatorname{AffCone}(X)$ does not admit linear $\GG_a$-actions.
Then, using Theorems \ref{ac} and \ref{mt}, we obtain the following corollary.

\begin{corollary}
Suppose that $X$ is a smooth Fano threefold of genus $g=9$ or $g=10$ with $\Pic(X)=\ZZ[-K_X]$. Then the affine cone $\operatorname{AffCone}(X)$ admits an effective $\GG_a$-action, and any such action is non-linear. In particular, the automorphism group $\Aut(\operatorname{AffCone}(X))$ of this cone is infinite-dimensional.
\end{corollary}

 To prove Theorem \ref{mt} we use the double projection from a line $L$ to show that a smooth Fano threefold of genus $9$ (respectively, $10$) with $\operatorname{Pic}(X)=\ZZ[-K_X]$ contains an open subset $U\subset X$ such that $U$ is isomorphic to the complement of a cubic surface $F_L\subset\PP^3$ (respectively, the complement of a quartic surface $F_L$ in a smooth quadric $Q\subset\PP^4$). Next, in Proposition \ref{ccc} and Proposition \ref{l2} we show that if the surface $F_L$ is singular, then the complement contains a cylinder, so $X$ contains a cylinder as well.

 On the other hand, in Corollary \ref{mp} we provide sufficient conditions for $F_L$ to be singular, i.e. we prove that if
 \begin{itemize}
 \item
 either the normal bundle $\NNN_{L/X}$ is isomorphic to $\OOO_{\PP^1}(1)\oplus\OOO_{\PP^1}(-2)$,
 \item
 or two other lines intersect $L$ at the same point,
 \item
 or the number of lines intersecting $L$ is less than for a general line,
 \end{itemize}
 then the surface $F_L$ is singular and, as we explained above, $X$ contains a cylinder.

 Finally, we show that for any Fano threefold $X$ of genus $9$ or $10$ there is a line $L$ satisfying one of the above properties.

 In the case of genus $9$, we argue by contradiction. Assuming that no lines on $X$ have either of the above properties, we prove in Proposition \ref{pl0} that the natural morphism from the curve $C$ parametrising intersecting pairs of lines on $X$ to the Hilbert scheme of lines on $X$ is \'{e}tale of degree $5$ and conclude that $\deg K_C = 160$. On the other hand, we check that the singular locus of the surface $R \subset X$ swept by the lines is covered by $C$ and this cover is \'{e}tale of degree $2$, and using the adjunction formula for the normalisation of $R$ we observe that $\deg K_C$ is divisible by $3$, thus arriving at a contradiction.

 In the case of genus $10$, we use a more straightforward argument proving the following theorem.

\begin{theorem}
\label{lof}
    Let $X$ be a smooth Fano threefold of genus $g=10$ with $\Pic(X)=\ZZ[-K_X]$. Then there is a point $P\in X$ such that the length of the Hilbert scheme $\mathrm{F}_1(X,P)$ of lines on~$X$ passing through~$P$ is equal to $3$. Moreover, if the Hilbert scheme of lines $\mathrm{F}_1(X)$ is smooth, then there are three distinct lines on $X$ meeting at the point $P$.
\end{theorem}

\begin{remark}
    Suppose that $X$ is a smooth Fano threefold of genus $g\geqslant 7$ with $\Pic(X)=\ZZ[-K_X]$. Then for any point $P\in X$ there are at most three lines having a common point $P$, i.e. the length of the Hilbert scheme $\mathrm{F}_1(X,P)$ is less than or equal to three for every $P\in X$ (see \cite[Corollary A.6]{kp}).
\end{remark}

The structure of the present paper is as follows. In Section \ref{pre} we recall two constructions of Fano threefolds with $\Pic(X)=\ZZ[-K_X]$. Section \ref{pre} also contains some general facts about the Hilbert schemes of lines of smooth Fano threefolds $X$ of genus $9$ and $10$ with $\Pic(X)=\ZZ[-K_X]$. In Section \ref{on} we show that the complement of any singular cubic surface in $\PP^3$ and the complement of any singular surface of degree $4$ in a smooth three-dimensional quadric are cylindrical. In Section \ref{cyl} we prove that the existence of three distinct lines having a common point on $X$ implies the cylindricity of $X$; we also prove that the existence of a line on $X$ intersecting less than $14-g$ other lines implies the cylindricity of $X$. In Section \ref{lin2} we show that $X$ has the desired property, i.e., in the case $g=9$, we prove that if the Hilbert scheme of lines is smooth, then either there are three distinct lines having a common point or there is a line intersecting less than five other lines and, in the case $g=10$, there are three lines on $X$ meeting at a common point.

The author is grateful to Yuri Prokhorov for stating the problem, useful discussions and constant attention to this work. The author thanks Alexander Kuznetsov, whose remarks allowed to significantly improve this paper, for helpful suggestions and useful conversations. Also, the author would like to thank Grigory Belousov for useful discussions.

\section{Preliminaries}
\label{pre}

In this section, we provide general information on Fano threefolds of genus $9$ and $10$. We also provide a few properties of lines on Fano threefolds of genus $9$ and $10$.

\begin{definition}
    A smooth Fano threefold $X$ is \emph{prime} if $\Pic(X)=\ZZ[-K_X]$.
\end{definition}

Recall that if $X$ is a smooth prime Fano threefold of genus $g=9$ or $g=10$, then $-K_X$ is very ample, so $X\subset \PP^{g+1}$.

The following theorem gives one description of Fano threefolds of genus $9$ and $10$. Historically, it was obtained after the 
Iskovskikh construction (Theorem \ref{ic}), but we provide it first, because it is more invariant.

\begin{theorem}[S.~Mukai {\cite{muk}}, {\cite{bay}}]
    \label{muk}
    Suppose $X=X_{2g-2}\subset \PP^{g+1}$ is a smooth anticanonically embedded prime Fano threefold of genus $g=9$ or $g=10$. Then $X$ is a transversal linear section of the $(15-g)$-dimensional homogeneous Fano variety $V_{2g-2}\subset \PP^{13}$. More precisely,
    \begin{align*}
    X_{16}&=V_{16}\cap \PP^{10}\subset \PP^{13}&\text{and}\qquad\qquad X_{18}&=V_{18}\cap \PP^{11}\subset \PP^{13},\quad\text{where}\\
    V_{16}&=\operatorname{SP}_6/P_3&\text{and}\qquad\qquad\hspace{3pt} V_{18}&=\G_2/P_2.
    \end{align*}
\end{theorem}

In the theorem $\operatorname{SP}_6$ and $\G_2$ denote the simple algebraic groups of types $\operatorname{C}_3$ and $\G_2$, respectively, and $P_3$ (respectively, $P_2$) denotes a maximal parabolic subgroup corresponding to the long simple root.

Further, we provide another construction of Fano threefolds of genus $9$ and $10$.

\begin{theorem}[{V.~A.~Iskovskikh \cite{isk3}}]
\label{ic}
    Suppose that $X=X_{2g-2}\subset \PP^{g+1}$ is a smooth anticanonically embedded prime Fano threefold of genus $g=9$ or $g=10$. Let $H\sim-K_X$ be the hyperplane class and let $L\subset X$ be a line. Denote by $$\psi=\phi_{|H-2L|}:X\dashrightarrow W\subset\PP^{g-6}$$ the double projection from the line $L$, i.e. the rational map given by the linear system $|H-2L|$.
    If $\varphi:\widehat X\to X$ is the blow-up of the line $L$, then there exists a Sarkisov link:
    \begin{equation*}
    \centerline{
        \xymatrix{
        &\widehat{F}\ar@{^{(}->}[rr]\ar[dl]&&\widehat
        X\ar[dl]_{\varphi}\ar@
        {-->}[rr]^{\chi}&
        & \widetilde X\ar[dr]^{\sigma}&&\widetilde{D}\ar@{_{(}->}[ll]\ar[dr]&
        \\
        L\ar@{^{(}->}[rr]&&X\ar@{-->}[rrrr]_{\psi}&&&&W&&\Gamma\ar@{_{(}->}[ll]
        }}
        \end{equation*}
        where $\chi$ is a flop and $\sigma$ is an extremal Mori contraction. Moreover, $W$ is a smooth Fano threefold and the morphism $\sigma:\widetilde{X}\to W$ is the blow-up of a smooth irreducible curve $\Gamma\subset W$ and $\widetilde{D}$ is its exceptional divisor. Next, let $\widehat{F}=\varphi^{-1}(L)$ be the exceptional divisor of $\varphi$, let $\widetilde{F}\subset\widetilde{X}$ be the proper transform of $\widehat{F}$, and let $F_L=\sigma(\widetilde{F})\subset W$.
    \begin{itemize} 
        \item If $g=9$, then $W=\PP^3$, $\Gamma\subset\PP^3$ is a non-hyperelliptic (see \cite[Remark 4.3.9. (ii)]{isk2}) curve of genus $3$ and degree $7$ lying on the cubic surface $F_L$.
        \item If $g=10$, then $W=Q\subset\PP^4$ is a smooth quadric, $\Gamma\subset W$ is a curve of genus $2$ and degree $7$ lying on the surface $F_L$ of degree $4$.  
    \end{itemize}
\end{theorem}

The subscript $L$ in the notation $F_L$ is there to emphasise that the corresponding surface depends on the line $L\subset X$. 

 Recall that $$\mathrm{F}_1(X)=\operatorname{Hilb}^{1+t}(X;-K_X)$$ denotes the Hilbert scheme of (anticanonical) lines of a smooth Fano variety $X$. We refer to \cite{kps} for a survey on Hilbert schemes of lines and conics on Fano threefolds of Picard rank one.

\begin{theorem}
\label{thl}
   Suppose that $X=X_{2g-2}\subset \PP^{g+1}$ is a smooth anticanonically embedded prime Fano threefold of genus $g\geqslant 3$. Then the following assertions hold:
 \begin{enumerate}
\item
$\mathrm{F}_1(X)\neq\varnothing$ and any irreducible component of $\mathrm{F}_1(X)$ has dimension $1$.
 \item
For the normal bundle of any line $L\subset X$, the following holds:
$$\NNN_{L/X}\cong
\begin{cases}
\OOO_{\PP^1}\oplus\OOO_{\PP^1}(-1)\qquad\text{or}
\\
\OOO_{\PP^1}(1)\oplus\OOO_{\PP^1}(-2)
\end{cases}$$
 \item
The scheme $\mathrm{F}_1(X)$ is smooth at a point $[L]\in
\mathrm{F}_1(X)$ if and only if the normal bundle $\NNN_{L/X}$ is isomorphic to $\OOO_{\PP^1}\oplus\OOO_{\PP^1}(-1)$.
 \item
If $g\geqslant 7$, then any line $L$ on $X$ intersects at most a finite number of other lines $L_i$ on $X$.
\item For $g\in\{9, 10\}$, the scheme $\mathrm{F}_1(X)$ is generically reduced.
\end{enumerate}
\end{theorem}

\begin{proof}
    The first assertion follows from \cite[Theorem 1.2]{sho} and \cite[Proposition 4.2.2(i)]{isk2}; and \cite[Lemma 4.2.1(i)]{isk2} implies the second assertion. The third assertion is precisely \cite[Proposition 4.2.2(ii)]{isk2} and the fourth assertion is precisely \cite[Proposition 3(iv)]{isk3}. The fifth assertion for $g=9$ and $g=10$ is proved in \cite{gru} and \cite{pro}, respectively.
\end{proof}

\begin{remark}[{\cite[Proposition 3(iv)]{isk3}}]
\label{rm3}
    In the notation of Theorem \ref{ic}, the flopped curves $\widehat{C}_i\subset\widehat{X}$ of $\chi$ are precisely the proper transforms of the lines intersecting $L$ and, if $\NNN_{L/X}\cong \OOO_{\PP^1}(1)\oplus\OOO_{\PP^1}(-2)$, the exceptional section of $\widehat{F}$.
\end{remark}

\begin{theorem}[{\cite[Theorem 2.13]{kis}}]
\label{lot2}
Suppose that $X=X_{2g-2}\subset \PP^{g+1}$ is a smooth anticanonically embedded prime Fano threefold of genus $g=9$ or $g=10$.
         Let $L\subset X$ be a line. Then the normal bundle $\NNN_{L/X}$ is isomorphic to $\OOO_{\PP^1}\oplus\OOO_{\PP^1}(-1)$ if and only if the corresponding surface $F_L$ (see the notation of Theorem \ref{ic}) is normal.
\end{theorem}

Recall from Theorem \ref{muk} that $V_{16}=\operatorname{SP}_6/P_3\subset\PP^{13}$ and $V_{18}=\G_2/P_2\subset\PP^{13}$, where $\operatorname{SP}_6$ and $\G_2$ denote the simple algebraic groups of types $\operatorname{C}_3$ and $\G_2$, respectively, and $P_3$ (respectively, $P_2$) denotes a maximal parabolic subgroup corresponding to the long simple root. Also recall that if $X_{2g-2}$ is a smooth anticanonically
embedded prime Fano threefold of genus $g=9$ or $g=10$, then $X_{2g-2}=V_{2g-2}\cap\PP^{g+1}$.

\begin{theorem}[{\cite[Theorem 4.2.7]{isk2}}]
\label{lot1}
Suppose that $X=X_{2g-2}\subset \PP^{g+1}$ is a smooth anticanonically embedded prime Fano threefold of genus $g=9$ or $g=10$. Then if $X$ is general, then the scheme $\mathrm{F}_1(X)$ is smooth and irreducible.
\end{theorem}

We shall provide a construction that is inverse to that of Theorem \ref{ic}. First, let us introduce some notation and known facts.

\begin{lemma}[{\cite[Lemma 2.1 and Lemma 2.4]{kis}}]
\label{lc}
\begin{itemize}
\item 
    Any smooth curve $\Gamma$ of genus $3$ and degree $7$ in $\PP^3$ lies on a unique irreducible cubic surface $F=F(\Gamma)$ in $\PP^3.$           
\item
    Suppose that $\Gamma$ is a smooth curve of genus $2$ and degree $7$ in $\PP^4$. Assume that $\Gamma$ is linearly non-degenerate, i.e. $\Gamma$ is not contained in any hyperplane of $\PP^4$. Then the quadrics containing $\Gamma$ form a pencil, and a general quadric of the pencil is smooth. The base locus of the pencil is an irreducible surface $F=F(\Gamma)$ of degree $4$ in $\PP^4.$
\end{itemize}
\end{lemma}

\begin{notation}
\label{nc}
Further, we will consider the following two cases:
\begin{itemize}
\item
    Case $g=9$. Let $W=\PP^3$ and let $\Gamma\subset W$ be a smooth non-hyperelliptic curve of genus $3$ and degree $7$.
\item 
    Case $g=10$. Let $W=Q\subset\PP^4$ be a smooth quadric and let $\Gamma$ be a smooth curve of genus $2$ and degree $7$ on $Q$.
\end{itemize}

In each case $F=F(\Gamma)$ denotes the corresponding surface of Lemma \ref{lc}.
\end{notation}

\begin{proposition}[{\cite[Proposition 2.2]{kis}}]
    \label{prd}
    In the assumptions of Notation $\ref{nc}$, the surface $F=F(\Gamma)\subset\PP^{g-6}$ belongs to one of the following classes:
\begin{itemize}
\item
    $F \subset \PP^{g-6}$ is a normal del Pezzo surface with at worst Du Val singularities.
\item
    $F \subset \PP^{g-6}$ is a ruled surface and the singular locus $\Lambda=\Sing(F)$ is a line. In particular, the surface $F$ is not normal.
\end{itemize}
\end{proposition}

\begin{theorem}[{\cite[Theorem 2.6]{kis}}]
\label{thc}
    In the assumptions of Notation \ref{nc}, there is a Sarkisov link:
\begin{equation*}
\centerline{
\xymatrix{
&\widetilde D\ar[dl]\ar@{^{(}->}[rr]&&\widetilde
X\ar[dl]_{\sigma}\ar@
{-->}[rr]^{\chi^{-1}}&
& \widehat X\ar[dr]^{\varphi}&&\widehat
F\ar[dr]\ar@{_{(}->}[ll]
\\
\Gamma\ar@{^{(}->}[rr]&&W\ar@{-->} [rrrr]_{\psi^{-1}}&&&&X&&
L\ar@{_{(}->}[ll] }}
\end{equation*}
where $\sigma$ is the blow-up of $\Gamma$, $\chi$ is a flop, $X=X_{2g-2}$ is a smooth prime Fano threefold of genus $g$ and $\varphi$ is the blow-up of a line $L$ on~$X$. The exceptional divisor $\widehat F$ of the morphism $\varphi$ is the proper transform of the surface $F=F(\Gamma)\subset W$, and $\widetilde{D}\subset\widetilde{X}$ is the exceptional divisor of the morphism $\sigma$. This link is inverse to that of Theorem \ref{ic}. Moreover, $L\subset D$, where $D$ is the proper transform of $\widetilde{D}$ in $X$.
\end{theorem}

\begin{remark}
    Note that if $\Gamma\subset\PP^3$ is a hyperelliptic curve of genus $3$ and degree $7$, then $-K_{\widetilde{X}}$ is not numerically effective, where $\widetilde{X}=\Bl_{\Gamma}(\PP^3)$ (see \cite[Remark 4.3.9(ii)]{isk2}), so the construction of Theorem \ref{thc} is not applicable in this case. Also note that for any (non-hyperelliptic) curve of genus $3$ there is an embedding of degree $7$ into $\PP^3$.
\end{remark}

\begin{remark}[{\cite[Proposition 2.12(4)]{kis}}]
\label{rm4}
     In the notation of Theorem \ref{thc}, each flopping curve $\widetilde{C}\subset\widetilde{X}$ is the proper transform of a $(13-g)$-secant line of $\Gamma$.
\end{remark}

\begin{corollary}[{\cite[Corollary 2.11]{kis}}]
\label{cc}
In the notation of Theorem \ref{thc}, we have
$$X\setminus D\simeq W\setminus F,$$ where $D\subset X$ is the proper transform of the divisor $\widetilde{D}$.
\end{corollary}
\begin{proof}
    Note that any of the flopping curves $\widetilde{C}\subset\widetilde{X}$ and any of the flopped curves $\widehat{C}\subset\widehat{X}$ lie on $\widetilde{F}$ and $\widehat{D}$, respectively, where $\widetilde{F}\subset\widetilde{X}$ and $\widehat{D}\subset\widehat{X}$ are the proper transforms of $\widehat{F}$ and $\widetilde{D}$, respectively. Indeed, by Remark \ref{rm3} and Remark \ref{rm4}, the flop $\chi$ is a $-\widehat{F}$-flop and the flop $\chi^{-1}$ is a $-\widetilde{D}$-flop (see e.g. \cite[Definition 1.4.13, Theorem 1.4.15 and Remarks 1.4.16]{isk2}), whence the intersection numbers $\widetilde{F}\cdot \widetilde{C}$ and $\widehat{D}\cdot\widehat{C}$ are negative, so the assertion follows. Hence $$\widehat{X}\setminus(\widehat{F}\cup\widehat{D})\simeq\widetilde{X}\setminus(\widetilde{F}\cup\widetilde{D}),$$ so $X\setminus \varphi (\widehat{D})\simeq W\setminus \sigma (\widetilde{F})$, because $\Gamma\subset \sigma (\widetilde{F})$ and $L\subset\varphi (\widehat{D})$. Thus, we obtain the assertion of the corollary.
\end{proof}

Using Corollary \ref{cc} we obtain that if the variety $W\setminus F$ is cylindrical, so the variety $X$ is. The case of a non-normal surface $F=F(\Gamma)$ was considered in \cite{kis}.

\begin{corollary}[see {\cite[Theorem 3.1]{kis}}]
    In the notation of Theorem \ref{thc}, if the surface $F$ is non-normal, then the variety $W\setminus F$ is cylindrical. In particular, by Theorem \ref{thl} and Theorem \ref{lot2}, if $\operatorname{F}_1(X)$ is singular, then $X$ is cylindrical.
\end{corollary}

\section{On cylindricity of \texorpdfstring{$W\setminus F$}{W\textbackslash F}}
\label{on}

In this section, we show that the complement of any singular cubic surface in $\PP^3$ and the complement of any singular surface of degree $4$ in a smooth three-dimensional quadric are cylindrical.

We follow the proof of Lemma $5.21$ of \cite{bla} and show that the following proposition holds.

\begin{proposition}[see {\cite[Lemma 5.21]{bla}}]
\label{ccc}
       Suppose that $F\subset\PP^3$ is a (possibly reducible) cubic surface. If the surface $F$ is singular, then the variety $\PP^3\setminus F$ is cylindrical. 
\end{proposition}
\begin{proof}
    Since the surface $F$ is singular, we may assume $F$ to be given by the equation
    $$f=f_2(x,y,z)w+f_3(x,y,z)=0,$$
    where $f_i$ are homogeneous polynomials of degree $i$ and $(w:x:y:z)$ are the corresponding homogeneous coordinates of $\PP^3$. Note that the point $P=(1:0:0:0)\in F$ is singular.
 
    If $f_2=0$, then $F$ is a cone over a plane cubic $C'\subset\PP^2$, hence $\PP^3\setminus F$ is the total space of the restriction of $\OOO_{\PP^2}(1)$ to the complement of the cubic $\PP^2\setminus C'$. Since $\OOO_{\PP^2}(1)$ is trivial on the complement of any line, it follows that $\PP^3\setminus(F \cup H)$ is a cylinder, where $H$ is a plane through the vertex of $F$.

    If $f_2\neq0$, then the assertion of the proposition follows from the proof of Lemma $5.21$ of~\cite{bla}:

Since $f_2\neq0$, we have the conic $$C = \{f_2(x,y,z)=0\}\subset\PP^2.$$ We can choose the coordinates $(x:y:z)$ 
	such that the line $L=\{z=0\}$ is contained in $C$ if $f_2$ is reducible, and that the line
	 $\{z=0\}$ is tangent to  $C$ if $f_2$ is irreducible. We will show that $\PP^2\setminus(C\cup L)$ is a cylinder. 

     If $C$ is reducible, then $\PP^2\setminus(C\cup L)=\PP^2\setminus C$ is a cylinder over $\AAA^1\setminus\{0\}$ or $\AAA^1$.
     
     If $C$ is irreducible, then the pencil of conics generated by $C$ and the double line $2L$ defines a rational map $\PP^2\dashrightarrow\PP^1$ such that $C$ and $L$ are fibres, hence $\PP^2\setminus(C\cup L)$ is a trivial $\AAA^1$-fibration over $\PP^1\setminus \{0,\infty \}$, therefore $\PP^2\setminus(C\cup L)$ is a cylinder over $\AAA^1\setminus\{0\}$.

     Let $q:\Bl_P\PP^3\to\PP^3$ be the blow-up of the singular point $P\in F$. Then $\widetilde{W}=\Bl_P\PP^3$ is isomorphic to $\PP_{\PP^2}(\OOO_{\PP^2}\oplus\OOO_{\PP^2}(1))$. Consider the following commutative diagram:
    \begin{equation*}
    \centerline{
\xymatrix{
    &E\ar[dl]\ar@{^{(}->}[rr]&&\widetilde{W}\ar[dl]_{q}\ar[dr]^{r}&&H'\cup H''\ar[dr]\ar@{_{(}->}[ll]
\\
P\ar@{^{(}->}[rr]&&\PP^3\ar@{-->} [rr]^{\pi}&&\PP^2&&
C\cup L\ar@{_{(}->}[ll]
        }
       }
    \end{equation*}
    where $\pi: \PP^3\dashrightarrow \PP^2$ is the projection from the singular point $P$ of the surface $F$. Let $E$ be the exceptional divisor of the morphism $q$, let $H'=r^{-1}(C)$, let $H''=r^{-1}(L)$, let $\widetilde{F}\subset\widetilde{W}$ be the proper transform of $F$ and let $H=q(H'')\subset\PP^3$ be a plane. Hence $$\PP^3\setminus (F\cup H)\simeq \widetilde{W}\setminus(\widetilde{F}\cup H''\cup E)\supset \widetilde{W}\setminus(\widetilde{F}\cup H'\cup H''\cup E).$$ 
    Note that $\widetilde{F}\cap r^{-1}(\PP^2\setminus(C\cup L))$ and $E\cap r^{-1}(\PP^2\setminus(C\cup L))$ are two disjoint sections of the restriction $r^{-1}(\PP^2\setminus(C\cup L))\to \PP^2\setminus(C\cup L)$ of the morphism $r$. Therefore, $$\widetilde{W}\setminus(\widetilde{F}\cup H'\cup H''\cup E)=r^{-1}(\PP^2\setminus(C\cup L))\setminus (\widetilde{F}\cup E)$$ is a cylinder, because $r^{-1}(\PP^2\setminus(C\cup L))\setminus (\widetilde{F}\cup E)$ is a trivial $\AAA^1\setminus\{0\}$-fibration over $\PP^2\setminus(C\cup L)$ and $\PP^2\setminus(C\cup L)$ is a cylinder. Thus, $\PP^3\setminus F$ is cylindrical.
    
  This completes the proof of the proposition.
 \end{proof}

\begin{remark}[{\cite[Corollary 4.2]{che}}]
    Note that in Proposition \ref{ccc} the non-smoothness condition of $F$ is necessary:
    If $F\subset\PP^3$ is a smooth cubic surface, then the variety $\PP^3\setminus F$ is not cylindrical. 
\end{remark}

We also prove a proposition similar to Proposition \ref{ccc}.

\begin{proposition}
\label{l2}
    Let $W\subset\PP^4$ be a smooth quadric and let $Q\subset\PP^4$ be a quadric such that the intersection $F=W\cap Q$ is a singular surface. Then the variety $W\setminus F$ is cylindrical.
\end{proposition}
\begin{proof}
    Consider the following commutative diagram:
    \begin{equation*}
    \centerline{
\xymatrix{
    &E\ar[dl]\ar@{^{(}->}[rr]&&\widetilde
W\ar[dl]_{q}\ar[dr]^{r}&&E_1\ar[dr]\ar@{_{(}->}[ll]
\\
P\ar@{^{(}->}[rr]&&W\ar@{-->} [rr]^{\pi}&&\PP^3&&
C\ar@{_{(}->}[ll]
        }
       }
    \end{equation*}
    where $\pi: W\dashrightarrow \PP^3$ is the projection from a singular point $P$ of the surface $F$ and the morphism $q:\Bl_P W=\widetilde{W}\to W$ is the blow-up of the point $P\in W$. Note that $r$ is the blow-up of a smooth conic~$C$. Let $E$ be the exceptional divisor of the morphism $q$, let $E_1$ be the exceptional divisor of the morphism $r$, and let $\widetilde{F}\subset\widetilde{W}$ be the proper transform of $F$. 
    
Let $T_P W\subset\PP^4$ be the embedded tangent space at the point $P$ of the quadric $W$. Since $W\cap T_P W$ is the union of the lines on $W$ containing the point $P$, it follows that $W\cap T_PW=q(E_1)$, hence $$W\setminus (F\cup T_P W)\simeq \widetilde{W}\setminus(\widetilde{F}\cup E\cup E_1).$$ The points of the divisor $E$ correspond to the tangent lines to $W$ containing the point $P$, therefore $r(E)=H$, where $H\subset \PP^3$ is the linear span of $C$. Then $$\widetilde{W}\setminus(\widetilde{F}\cup E\cup E_1)\simeq \PP^3\setminus(H\cup S),$$ where $S=r(\widetilde{F})$. Moreover, $$S\cdot H^2=\widetilde{F}\cdot (r^*H)^2=(q^*F-kE)\cdot(r^*H)^2=q^*F\cdot(r^*H)^2-k=4-k,$$ where $k\geqslant 2$ is the multiplicity of $F$ at $P$.
    
    Thus, we obtain that $H\cup S\subset\PP^3$ is contained in a reducible cubic surface, therefore $\PP^3\setminus(H\cup S)$ is cylindrical by Proposition \ref{ccc}.
\end{proof}

\section{Cylindricity of Fano threefolds of genus \texorpdfstring{$9$}{9} and \texorpdfstring{$10$}{10}}
\label{cyl}

 In this section, we prove that if either the normal bundle $\NNN_{L/X}$ of a line $L$ on a smooth prime Fano threefold $X$ of genus $g=9$ or $g=10$ is isomorphic to $\OOO_{\PP^1}(1)\oplus\OOO_{\PP^1}(-2)$ or two other lines intersect $L$ at the same point or the number of lines intersecting $L$ is less than $14-g$ other lines, then the surface $F_L$ is singular, in particular, $X$ is cylindrical.

Let us recall the definition of the Atiyah--Kulikov flop.

\begin{definition}
    Let $\widetilde{X}$ be a smooth threefold and let $\widetilde{C}\subset \widetilde{X}$ be a smooth rational curve on $\widetilde{X}$ such that $$\NNN_{\widetilde C/\widetilde X}\cong\OOO_{\PP^1}(-1)\oplus\OOO_{\PP^1}(-1).$$ Denote by $q:\Bl_{\widetilde C}\widetilde X=\widebar{X}\to\widetilde X$ the blow-up of the curve $\widetilde C$ and by $E=\PP^1\times\PP^1$ the exceptional divisor of the morphism $q$. Then a birational map $\chi:\widetilde X\dashrightarrow \widehat X$ is called an \emph{Atiyah--Kulikov flop} if there exists the following commutative diagram:
    \begin{equation*}
     \centerline{
\xymatrix{
    &\widebar{X}\ar[dl]_{q}\ar[dr]^{r}
\\
\widetilde{X}\ar@{-->} [rr]^{\chi}&&\widehat{X}
        }
       }
       \end{equation*}
    where $\widehat X$ is a smooth threefold, $r:\widebar{X}\to \widehat X$ is the blow-up of the smooth rational curve $\widehat C=r(E)$, and the divisor $E$ is contracted in two different ways, i.e. the morphisms $q|_E:E\to \widetilde C$ and $r|_E:E\to \widehat C$ correspond to the two distinct projections of $\PP^1\times\PP^1$ onto $\PP^1$. Notice that $\widetilde C\subset\widetilde X$ and $\widehat C\subset\widehat X$ are called a \emph{flopping curve} and a \emph{flopped curve}, respectively.
    \end{definition}
    
    Note that the Atiyah--Kulikov flops exist in the category of smooth Moishezon manifolds (see \cite{ati} or \cite{kul}).

We will need the following lemma.

\begin{lemma}
    \label{lf}
    Let $\widetilde X$ be a smooth threefold and let $\chi:\widetilde{X}\dashrightarrow\widehat{X}$ be a flop. Assume that $\chi$ is an Atiyah--Kulikov flop
    near each flopping curve. Suppose that $\widetilde{F}\subset\widetilde{X}$ is a smooth surface containing the flopping centre of $\chi$. Let $\widebar{F}\subset\widebar{X}$ be the proper transform of $\widetilde{F}$. Then $\widebar{F}$ is naturally isomorphic to $\widetilde{F}$. In particular, the map $$\chi|_{\widetilde{F}}:\widetilde{F}\dashrightarrow\widehat{X}$$ is regular on $\widetilde{F}$. Let $\widetilde C\subset\widetilde X$ be a flopping curve, let $\widehat C\subset\widehat X$ be the corresponding flopped curve and let $\widehat F=\chi_*(\widetilde F)\subset \widehat X$. If $\NNN_{\widetilde C/\widetilde F}\cong\OOO_{\PP^1}(-1)$, then $\widehat{C}\cdot\widehat{F}=1$, $\widehat{C}\not\subset\widehat{F}$ and the morphism $\chi|_{\widetilde{F}}:\widetilde{F}\to\widehat{F}$ is the contraction of the curve $\widetilde{C}$ near $\widetilde{C}$.
\end{lemma}
\begin{proof}
 Restricting to an open subset of $\widetilde{X}$, we may assume that the flopping centre of $\chi$ is a rational curve $\widetilde{C}$.
       Let $E\simeq\PP^1\times\PP^1$ be the exceptional divisor of the morphism $q$, where $q:\widebar{X}=\Bl_{\widetilde{C}}\widetilde{X}\to\widetilde{X}$. Since $q|_{\widebar{F}}:\widebar{F}\to\widetilde{F}$ is the blow-up $\Bl_{\widetilde{C}}\widetilde{F}\to\widetilde{F}$ and $\widetilde{C}\subset\widetilde{F}$ is a Cartier divisor, we see that the morphism $q$ maps the surface $\widebar{F}$ isomorphically onto the surface $\widetilde{F}$ and the curve $\widebar{C}=\widebar{F}\cap E$ isomorphically onto the curve $\widetilde{C}$. Thus, the map $\chi|_{\widetilde{F}}:\widetilde{F}\dashrightarrow\widehat{X}$ is regular on $\widetilde{F}$.
    
    Let us prove the second part of the lemma. Since $$\NNN_{\widetilde C/\widetilde F}\cong \OOO_{\PP^1}(-1)\subset \OOO_{\PP^1}(-1)\oplus\OOO_{\PP^1}(-1)\cong\NNN_{\widetilde C/\widetilde X},$$ it follows that $$\widebar C\cong\PP(\NNN_{\widetilde C/\widetilde F})\subset\PP(\NNN_{\widetilde C/\widetilde X})\cong \PP^1\times\PP^1$$ is a fibre of the second projection $p_2:\PP^1\times\PP^1\to\PP^1$. Let $f$ be a fibre of the first projection, i.e. a fibre of the morphism $q|_E:E\to \widetilde C$. We see that $\widehat{C}\not\subset\widehat{F}$, because $r(f)=\widehat{C}$, $r(\widebar{F})\cap \widehat{C}=r(\widebar{F}\cap E)=r(\widebar{C})$ and $r(\widebar{C})$ is a point, where $r:\widebar{X}=\Bl_{\widehat{C}}\widehat{X}\to\widehat{X}$, whence the morphism $\chi|_{\widetilde{F}}:\widetilde{F}\to\widehat{F}$ is the contraction of the curve $\widetilde{C}$ near $\widetilde{C}$. Therefore $\widebar F\sim r^{*}(\widehat F)$, where $r^{*}(\widehat F)$ is the pullback of $\widehat{F}$. Thus, $$\widehat F\cdot\widehat C=\widehat F\cdot r_*(f)=r^*(\widehat F)\cdot f=\widebar F\cdot f=1.$$ This completes the proof of the lemma. 
\end{proof}

The following proposition helps us to describe the surface $F_L$ corresponding to a line $L\subset X$ whenever the surface $F_L$ is smooth.

\begin{proposition}
\label{rm}
Suppose that $X$ is a smooth prime Fano threefold of genus $g=9$ or $g=10$. Let $L\subset X$ be a line and let $F_L$ be as in the notation of Theorem \ref{ic}. Assume that the surface $F_L$ is smooth. Then $F_L$ is isomorphic to the blow-up of $\widehat{F}$ at the points $\widehat{F}\cap \widehat{L_i}$, where $\widehat{L_i}\subset\widehat{X}$ are the proper transforms of the lines intersecting $L$. Moreover, $\NNN_{L/X}\cong\OOO_{\PP^1}\oplus\OOO_{\PP^1}(-1)$, there are exactly $14-g$ lines on $X$ intersecting $L$ and their intersection points with $L$ are distinct.
\end{proposition}

\begin{proof}
  Since $\Gamma\subset F_L$ is a Cartier divisor, it follows that the surface $\widetilde{F}\subset \widetilde{X}$ is isomorphic to $F_L$, in particular, $\widetilde{F}$ is a smooth del Pezzo surface. 

    Since distinct irreducible flopping curves do not intersect each other (see \cite[Proposition 4(6), Theorem 13]{cut} and \cite[Corollary 5.6]{rei}), we see that to describe the flop $\chi^{-1}$ it is sufficient to describe the flop $\chi$ in neighbourhoods of the irreducible flopping curves.
    
    Let $\widetilde{C}\subset \widetilde{X}$ be an irreducible flopping curve. Since, by \cite[Lemma 2(iii)]{isk3}, $(-K_{\widetilde{X}}-\widetilde{F})\cdot\widetilde{C}=1$, we see that $$-\widetilde{F}\cdot\widetilde{C}=(-K_{\widetilde{X}}-\widetilde{F})\cdot\widetilde{C}=1.$$ In particular $\widetilde{C}\subset\widetilde{F}$, so, by Remark \ref{rm4}, $\sigma(\widetilde{C})$ is a line on $\sigma(\widetilde{F})$. Consider the short exact sequence of normal bundles. 
    $$0\longrightarrow \NNN_{\widetilde{C}/\widetilde{F}}\longrightarrow \NNN_{\widetilde{C}/\widetilde{X}}\longrightarrow\NNN_{\widetilde{F}/\widetilde{X}}|_{\widetilde{C}}\longrightarrow 0.$$ Since $\widetilde{C}$ is a $(-1)$-curve on $\widetilde{F}$ and $\widetilde{F}\cdot \widetilde{C}=-1$, it follows that this exact sequence is of the following form.
    $$0\longrightarrow\OOO_{\PP^1}(-1)\longrightarrow \OOO_{\PP^1}(a)\oplus\OOO_{\PP^1}(-2-a)\longrightarrow\OOO_{\PP^1}(-1)\longrightarrow 0,$$ where $a\in\ZZ_{\geqslant -1}$. Indeed, $\NNN_{\widetilde{C}/\widetilde{F}}\cong \OOO_{\widetilde{C}}(\widetilde{C})\cong\OOO_{\PP^1}(-1)$ and $\NNN_{\widetilde{F}/\widetilde{X}}|_{\widetilde{C}}\cong\OOO_{\widetilde{C}}(\widetilde{F})\cong\OOO_{\PP^1}(-1)$. But $$\operatorname{Ext}^1(\OOO_{\PP^1}(-1),\OOO_{\PP^1}(-1))=\operatorname{H}^1(\PP^1,\OOO_{\PP^1})=0,$$ therefore the short exact sequence of normal bundles splits. Thus, $a=-1$.
    
    From this we see that in a neighbourhood of $\widetilde{C}$ the flop $\chi$ is the Atiyah--Kulikov flop (see \cite{ati} or \cite{kul}). Hence, by Lemma \ref{lf}, the birational map $$\chi^{-1}|_{\widetilde{F}}:\widetilde{F}\dashrightarrow \widehat{F}$$ is, in fact, regular and it is the contraction of the flopping curves. 
    
    Since $F_L$ is smooth, it follows, by Theorem \ref{lot2} and Remark \ref{rm3}, that the flopped curves are the proper transforms of the lines intersecting $L$. Note that the morphism $\chi^{-1}|_{\widetilde{F}}:\widetilde{F}\to \widehat{F}$ is the contraction of the flopping curves and every flopping curve $\widetilde{C}\subset\widetilde{F}$ is a $(-1)$-curve. Hence $\chi^{-1}|_{\widetilde{F}}:\widetilde{F}\to \widehat{F}$ is the blow-up of $\widehat{F}$ at the points $\widehat{F}\cap \widehat{L_i}$, where $\widehat{L_i}\subset\widehat{X}$ are the proper transforms of the lines intersecting $L$. 
    Since $\widetilde{F}$ is a del Pezzo surface, it follows that $\widehat{F}$ is also a del Pezzo surface. Hence $\NNN_{L/X}\cong\OOO_{\PP^1}\oplus\OOO_{\PP^1}(-1)$, because $\widehat{F}$ is a $\PP^1$-bundle over $L$. In particular, the number of flopped curves is equal to $K^2_{\widehat{F}}-K^2_{\widetilde{F}}=14-g$ and $\chi^{-1}|_{\widetilde{F}}:\widetilde{F}\to \widehat{F}$ is the blow-up of $14-g$ points, where no pair of points lies on a fibre of the $\PP^1$-bundle $\widehat{F}\to L$. 
    
    Thus, the flopped curves are $14-g$ disjoint curves, whose images in $X$ intersect $L$ transversely at distinct points. This completes the proof of the proposition.
\end{proof}

Using Proposition \ref{rm}, we obtain the following corollary.

\begin{corollary}
    \label{mp}
    Suppose that $X$ is a smooth prime Fano threefold of genus $g=9$ or $g=10$. Let $L\subset X$ be a line and let $F_L$ be as in the notation of Theorem \ref{ic}. Assume that one of the following holds.
    \begin{enumerate}
        \item the normal bundle $\NNN_{L/X}$ is isomorphic to $\OOO_{\PP^1}(1)\oplus\OOO_{\PP^1}(-2)$.
        \item two other lines intersect $L$ at the same point.
        \item the number of lines intersecting $L$ is less than $14-g$.
    \end{enumerate}
    Then the surface $F_L\subset W$ is singular, in particular, $X$ is cylindrical. 
\end{corollary}

\section{Lines on Fano threefolds of genus \texorpdfstring{$9$}{9} and \texorpdfstring{$10$}{10}}
\label{lin2}

In this section, we prove Theorem \ref{lof}, i.e. we prove that for any smooth prime Fano threefold $X$ of genus $10$ there are three lines (some of them may coincide) having a common point. We also prove Proposition \ref{pl0}, i.e. we prove that if $X$ is a smooth prime Fano threefold of genus $9$ with the smooth Hilbert scheme of lines, then either there are three distinct lines having a common point or there is a line intersecting less than five other lines. We prove Proposition \ref{pl0} and Theorem \ref{lof} using different methods. For the case $g=9$, we show that the incidence curve $C\subset U$ of $X$, i.e. the preimage of the locus of points of $X$ through which there pass at least two lines in the universal family of lines $U$, cannot be \'etale over $\operatorname{F}_1(X)$. For the case $g=10$, we use Mukai's description $V_{18}\cap\PP^{11}$ of Fano threefolds and a description of the Hilbert scheme of lines on $V_{18}$ to prove that there is a point $P\in X$ such that the intersection $\langle \mathcal{L}_P\rangle\cap \PP^{11}$ is of dimension three, where $\mathcal{L}_P$ is the union of the lines on $V_{18}$ through $P$. The former method cannot be applied for the case $g=10$. And we do not apply the latter method for the case $g=9$, because similar calculations are much more involved.

Recall that $\mathrm{F}_1(X)=\operatorname{Hilb}^{1+t}(X;-K_X)$ denotes the Hilbert scheme of lines of a smooth Fano variety $X$ and $\mathrm{F}_1(X,P)\subset \mathrm{F}_1(X)$ denotes the subscheme of lines through $P$.

It is well known that $V_{16}=\operatorname{SP}_6/P_3\subset\PP^{13}$ (see Theorem \ref{muk}) is isomorphic to the Lagrangian Grassmannian $\operatorname{LGr(3,6)}$. It is clear that the isotropic partial flag variety $\operatorname{IFl}(2,3,6)$ is isomorphic to $\PP(\mathcal{S}^*)$ over $\operatorname{LGr(3,6)}$, where $\mathcal{S}$ is the tautological vector bundle on $\operatorname{LGr(3,6)}$. Moreover, the projection $\operatorname{IFl}(2,3,6)\to\operatorname{IGr(2,6)}$ is the universal family of lines on $\operatorname{LGr(3,6)}$.

In order to prove Proposition \ref{pl0}, we will need the following lemma.

\begin{lemma}
\label{lg}
	Let $X$ be a smooth prime Fano threefold of genus $9$. Then the Hilbert scheme of lines $\operatorname{F}_1(X)$ is connected and its arithmetic genus $p_a(\operatorname{F}_1(X))$ is equal to $17$.
\end{lemma}

\begin{proof}
 The connectedness of $\operatorname{F}_1(X)$ is proved by the same argument as in \cite[Lemma 2.13]{kd}.
 
 Now, we will prove the second assertion. Consider the Hilbert scheme of lines $\operatorname{F}_1(\operatorname{LGr}(3,6))$, which is isomorphic to $\operatorname{IGr}(2,6)$. Then, since $\operatorname{LGr}(3,6)\cong V_{16}\subset\PP^{13}$, we have an embedding $\operatorname{IGr}(2,6)\subset\operatorname{Gr}(2,14)$. If $X$ is a smooth prime Fano threefold of genus $9$, then $$X=\operatorname{LGr}(3,6)\cap \PP^{10}\subset\PP^{13}$$ and $\operatorname{F}_1(X)$ is the scheme of zeros of a section $s\in \operatorname{H}^0(\operatorname{IGr}(2,6),(\mathcal{U}^*)^{\oplus 3})$, where $\mathcal{U}$ is the restriction of the tautological vector bundle on $\operatorname{Gr}(2,14)$ onto $\operatorname{IGr}(2,6)$. Since $\dim \operatorname{IGr}(2,6)=7$, the rank of $(\mathcal{U}^*)^{\oplus 3}$ equals $6$ and $\operatorname{F}_1(X)$ is of pure dimension $1$, we see that the section $s$ is regular. Then we have the Koszul resolution $$\cdots\to\Lambda^2(\mathcal{U}^{\oplus 3})\to\mathcal{U}^{\oplus 3}\to\OOO_{\operatorname{IGr(2,6)}}\to\OOO_{\operatorname{F}_1(X)}\to 0,$$ hence the Euler characteristic of $\OOO_{\operatorname{F}_1(X)}$ is the same for each smooth prime Fano threefold $X$ of genus $9$. In particular, by \cite[Theorem 4.2.7]{isk2}, $p_a(\operatorname{F}_1(X))=17$.
\end{proof}

\begin{proposition}
    \label{pl0}
     Suppose that $X$ is a smooth prime Fano threefold of genus $g=9$ with the smooth Hilbert scheme of lines. Then one of the following holds:
     \begin{enumerate}
         \item There are three distinct lines on $X$ having a common point.
         \item There is a line $L\subset X$ intersecting less than five other lines.
     \end{enumerate}
\end{proposition}
\begin{proof}
    Assume the converse that for every point $Q\in X$ there are at most two lines meeting at $Q$ and each line $L\subset X$ intersects exactly five other lines.

Let $f:U\to \mathrm{F}_1(X)$ be the universal family of lines on $X$. There is a natural morphism $\varepsilon: U\to X$, where $\varepsilon(U)=R$ is equivalent to $4H$ (\cite[Theorem 4.2.7]{isk2}). Moreover, $\varepsilon: U\to X$ is unramified, i.e. the differential $d\varepsilon:\NNN_{L/U}\to \NNN_{L/X}$ is non-degenerate at all points of $L$ for each line $L\subset X$. Indeed, the differential is the evaluation map  $\operatorname{H}^0(L,\NNN_{L/X})\otimes\OOO_L\to \NNN_{L/X}$, hence the differential $d\varepsilon:\OOO_L\to\OOO_L\oplus\OOO_L(-1)$ is non-degenerate at all points of $L$, therefore, $\varepsilon$ is unramified.

Consider the Hilbert scheme of conics $\operatorname{F}_2(X)$ and the divisor $\operatorname{F}_2^s(X)\subset \operatorname{F}_2(X)$ of singular conics. Since this divisor is non-empty (see e.g. \cite[Lemma 7.7]{kp}), it follows that it is proper of pure dimension $1$. Since $\operatorname{F}_1(X)$ is smooth, we see that all conics are reduced (see \cite[Remark 2.1.7]{kps}). Hence there is an \'{e}tale double cover $C\to \operatorname{F}_2^s(X)$ and a morphism $\eta:C\hookrightarrow U$, such that distinct lines $L$ and $L'$ on $X$ intersect at point $P$ if and only if there is a point $c\in C$ such that $\eta(c)=([L],P)$ and $\eta(\theta(c))=([L'],P)$, where $\theta$ is the involution of the double cover $C\to \operatorname{F}_2^s(X)$.
 
 In particular, the locus of points $P\in X$ such that the fibre $\varepsilon^{-1}(P)$ is non-trivial coincides with the image of the morphism $\varepsilon\circ\eta:C\to X$, so it is of pure dimension $1$, hence the morphism $\varepsilon:U\to X$ is birational onto its image and $C'=\varepsilon(\eta(C))=\Sing(R)$. 
 
 Since $\operatorname{F}^s_2(X)$ is a Cartier divisor on the smooth surface $\operatorname{F}_2(X)$ (see \cite[Proposition 2.3.6]{kps}), we see that $\operatorname{F}^s_2(X)$ is a Cohen--Macaulay curve, so its \'{e}tale double cover $C$ is also Cohen--Macaulay. Furthermore, the morphism $C\to\operatorname{F}_1(X)$ is flat by miracle flatness and its degree is $5$, because, by our assumptions, any line intersects exactly $5$ other lines and the intersection points are distinct, and the degree is equal to the number of intersection points for a general line. Set-theoretically, the fibre of the morphism $C\to\operatorname{F}_1(X)$ over a point $[L]\in\operatorname{F}_1(X)$ is the set of pairs $([L],P)$ such that there exists a line $L'\neq L$ with $L\cap L'=P$. If the cardinality of this set is $5$, then the scheme fibre is reduced, so the morphism is \'{e}tale over $[L]$. Hence, by our assumptions, the morphism $C\to\operatorname{F}_1(X)$ is \'{e}tale. Thus, the morphism $f\circ\eta:C\to \mathrm{F}_1(X)$ is \'{e}tale of degree $5$.

Since $C\to \mathrm{F}_1(X)$ is \'{e}tale and $\mathrm{F}_1(X)$ is smooth, the curve $C'=\varepsilon(\eta(C))=\Sing(R)$ is smooth. Indeed, since $\varepsilon: U\to X$ is unramified, it follows that the morphism $\varepsilon\circ\eta:C\to C'$ is. Consider the normalisation $\nu:\widetilde{C'}\to C'$ of $C'$. By the universal property of normalisation, there is a natural morphism $\psi:C\to\widetilde{C'}$ such that $\nu\circ\psi=\varepsilon\circ\eta$. Therefore, we see that the morphism $\nu$ is unramified and, by our assumptions, $\nu$ is bijective. Hence the normalisation morphism $\nu$ is an isomorphism and $C'$ is smooth. In particular, the morphism $\varepsilon\circ\eta:C\to C'$ is \'{e}tale of degree $2$.

Note that the surface $R$ is obtained by glueing the pairs of points on $U$ corresponding to the fibres of the morphism $\varepsilon\circ\eta:C\to C'$. Consider the blow-up $h:Y=\Bl_{C'}(X)\to X$ of the curve $C'$. Then $$K_Y\sim h^*K_X+E\quad\text{and}\quad\widetilde{R}\sim h^*R-2E,$$ where $\widetilde{R}$ is the proper transform of $R$, $E$ is the exceptional divisor of $h$ and $\widetilde{R}\cong U$. Let us compute $\deg K_C$ using two ways. First, since $f\circ\eta:C\to \mathrm{F}_1(X)$ is \'{e}tale of degree $5$ and, by Lemma \ref{lg}, $\mathrm{F}_1(X)$ is of genus $17$ and connected, it follows that $\deg K_C=32\cdot 5=160$, in particular, $\deg K_C$ is not divisible by $3$. Second, by the adjunction formula, $$K_{\widetilde{R}}\sim(K_Y+\widetilde{R})|_{\widetilde{R}}.$$ Therefore, $$K_{\widetilde{R}}\sim(-h^*H+E+4h^*H-2E)|_{\widetilde{R}}\sim(3h^*H-E)|_{\widetilde{R}}\sim 3h^*H|_{\widetilde{R}}-C,$$ thus, $$\deg K_C=( K_{\widetilde{R}}+C)\cdot C=3h^*H\cdot C$$
and we see that $\deg K_C$ is divisible by $3$, a contradiction, so the assertion follows.
\end{proof}

Our next goal is to prove Theorem \ref{lof}.
First, we need to compute some Chern numbers of smooth prime Fano threefolds of genus~$10$.

\begin{lemma}
\label{lcn}
    Suppose that $X$ is a smooth prime Fano threefold of genus $g=10$. Then $$c_1(X)\cdot (-K_X)^2=18,\quad c_2(X)\cdot (-K_X)=24\quad \text{and}\quad c_3(X)=0.$$
\end{lemma}
\begin{proof}
    Let $H$ denote $-K_X$. Since $c_1(X)\cdot H^2=H^3$, it follows that $c_1(X)\cdot H^2=18$.
    Since the top Chern class of the tangent bundle is equal to the topological Euler characteristic, we have $c_3(X)=0$ (see e.g. \cite[\S 12.2]{isk2}).
    
    Next, let $Y$ be a general hyperplane section of $X$, then $Y$ is a $K3$ surface, therefore $c_1(Y)=0$, $\NNN_{Y/X}\cong\OOO_Y(H)$ and $c_{2}(Y)=\chi_{\mathrm{top}}(Y)=24$. Consider the following short exact sequence: 
    $$0\longrightarrow\operatorname{T}Y\longrightarrow\operatorname{T}X|_Y\longrightarrow\NNN_{Y/X}\longrightarrow0,$$ therefore $c(Y)\cdot (1+tH)=c(X)|_Y$, thus, $c_2(X)\cdot H=24$.  
\end{proof}

\begin{proposition}[Theorem \ref{lof}]
\label{thp2}
    Suppose that $X$ is a smooth prime Fano threefold of genus $g=10$. Then there is a point $P$ on $X$ such that the length of the Hilbert scheme $\mathrm{F}_1(X,P)$ is equal to $3$. In particular, if $\mathrm{F}_1(X)$ is smooth, then there are three distinct lines through $P$.
\end{proposition}
\begin{proof}
    Let $V_{18}=\G_2/P_2\subset\PP^{13}$ (see Theorem \ref{muk}) and let $\operatorname{Ad:\G_2\to\GL(\mathfrak{g}_2)}$ be the adjoint representation of $\G_2$. We see that the restriction of 
$\operatorname{Ad}$ to $P_2$ naturally induces the $5$-dimensional representation $\rho:P_2\to\GL(\mathfrak{g}_2/\mathfrak{p}_2)$ and the associated homogeneous bundle 
$\G_2\times_{P_2}(\mathfrak{g}_2/\mathfrak{p}_2)$ is isomorphic to the tangent bundle $\operatorname{T}V_{18}$. Note that the representation $\rho$ has a unique irreducible $4$-dimensional subrepresentation. Therefore, its associated vector bundle is the unique homogeneous subbundle $\mathcal{E}$ of the tangent bundle 
$\operatorname{T}V_{18}$ of rank $4$ on $V_{18}$. Note that $\mathcal{E}$ is a contact structure on $V_{18}$. Then the homogeneous line bundle $\operatorname{T}V_{18}/\mathcal{E}$ is $\OOO(1)$.
Since $X=V_{18}\cap\PP^{11}\subset\PP^{13}$, it follows that there are the following two short exact sequences of vector bundles: 
    \begin{eqnarray*}
&&0\longrightarrow\operatorname{T}X\overset{f}{\longrightarrow}\operatorname{T}V_{18}|_X\longrightarrow \NNN_{X/V_{18}}\longrightarrow 0,  
\\
&&0\longrightarrow\mathcal{E}|_X\longrightarrow\operatorname{T}V_{18}|_X\overset{g}\longrightarrow \OOO_X(1)\longrightarrow 0.
\end{eqnarray*}

    Note that the union of the lines $\mathcal{L}_P\subset T_PV_{18}$ through any point $P\in V_{18}$ on $V_{18}$ is the cone over a twisted cubic curve with the vertex $P$ (see \cite[Proposition 1]{hwa}), where $T_PV_{18}$ is the embedded tangent space at the point $P$ of the variety $V_{18}$. Since the action of $P_2$ on $V_{18}$ preserves the linear spans $\langle\mathcal{L}_P\rangle$ and $\dim\langle\mathcal{L}_P\rangle\,=4$, it follows that $\mathcal{E}_P\subset \operatorname{T}_PV_{18}$ naturally corresponds to $\langle\mathcal{L}_P\rangle\,\subset T_PV_{18}$.
     
     If there is a point $Q\in X$ such that $$\dim(\operatorname{T}_QX\cap\mathcal{E}_Q)=3,$$ then the length of the Hilbert scheme $\mathrm{F}_1(X,Q)$ is equal to $3$. This implies that to prove that there are three lines (some of them may coincide) having a common point, it is sufficient to show that there is a point such that the morphism $g\circ f:\operatorname{T}X\to \OOO_X(1)$ vanishes. Assume the converse that there is no such point. Then $\mathcal{V}=\operatorname{T}^*X/\OOO_X(-1)$ is a vector bundle of rank $2$, and hence $c_3(\mathcal{V})=0$. On the other hand, using Lemma \ref{lcn}, we have 
     \begin{eqnarray*}
     c_3(\mathcal{V})=\{(1-c_1(X)+c_2(X)-c_3(X))/(1-H)\}_3&=&\\
     H^3-c_1(X)\cdot H^2+c_2(X)\cdot H-c_3(X)=c_2(X)\cdot H&=&24\neq0,
     \end{eqnarray*}
     a contradiction.
     
     Moreover, if $\mathrm{F}_1(X)$ is smooth, then the corresponding lines must be distinct (see \cite[Lemma 2.1.1]{kps} and Theorem \ref{thl}). This
completes the proof of the proposition.
\end{proof}

We provide the proof of Theorem \ref{mt}:

\begin{proof}[Proof of Theorem \ref{mt}]
     It follows from Theorem \ref{thl}, Corollary \ref{mp}, Proposition \ref{pl0} and Proposition \ref{thp2} that there is a line $L$ on $X$ such that the corresponding surface $F_L$ is singular. By Corollary \ref{cc} we have $$X\setminus D\simeq W\setminus F_L,$$ so the variety $X$ is cylindrical if the variety $W\setminus F_L$ is. We get from Proposition \ref{ccc} and Proposition \ref{l2} that the variety $W\setminus F_L$ is cylindrical. Thus, $X$ is cylindrical.
\end{proof}

\end{document}